\documentclass[12pt]{amsart}
\usepackage{amsmath}
\usepackage{amssymb}
\usepackage{latexsym}

\newcommand{\halmos}{\rule{5pt}{5pt}}

\numberwithin{equation}{section}

\newtheorem{prop}{\bf Proposition}[section]
\newtheorem{thm}[prop]{\bf Theorem}

\theoremstyle{definition}

\begin{document}
\title[Variants of confluent $q$-hypergeometric equations]
{Variants of confluent $q$-hypergeometric equations}
\author{Ryuya Matsunawa}
\address{Address of R.M.,~T.S., Department of Mathematics, Faculty of Science and Engineering, Chuo University, 1-13-27 Kasuga, Bunkyo-ku, Tokyo 112-8551, Japan}
\author{Tomoki Sato}
\author{Kouichi Takemura}
\address{Address of K.T., Department of Mathematics, Ochanomizu University, 2-1-1 Otsuka, Bunkyo-ku, Tokyo 112-8610, Japan}
\email{takemura.kouichi@ocha.ac.jp}
\begin{abstract}
Variants of the $q$-hypergeometric equation were introduced in our previous paper with Hatano.
In this paper, we consider degenerations of the variant of the $q$-hypergeometric equation, which is a $q$-analogue of confluence of singularities in the setting of the differential equation.
We also consider degenerations of solutions to the $q$-difference equations.
\end{abstract}
\maketitle

\section{Introduction}

The special functions have rich mathematical structures, and some of them have been applied to physics.
Gauss' hypergeometric function
\begin{align}
& \: _2 F _1 (\alpha ,\beta ;\gamma ;z)=1+ \frac{\alpha \beta  }{ \gamma } z + \frac{\alpha (\alpha +1) \beta (\beta +1 ) }{2! \: \gamma (\gamma +1)} z^2 +  \cdots  + \frac{(\alpha )_n (\beta )_n }{n!(\gamma )_n } z^n  + \cdots 
\end{align}
is a typical example of the special functions.
Here $(\lambda )_n= \lambda (\lambda +1) \dots (\lambda +n-1) $.
It is essentially characterized by Gauss' hypergeometric equation
\begin{equation}
 z(1-z) \frac{d^2y}{dz^2} + \left( \gamma - (\alpha + \beta +1)z \right) \frac{dy}{dz} -\alpha \beta  y=0.
\label{eq:GaussHGE}
\end{equation}
It is a second order Fuchsian differential equation with three singularities $\{ 0,1,\infty \}$.
Here the Fuchsian differential equation is a linear differential equation whose singularity on the Riemann sphere is always the regular singularity.
However, several special functions are related with the differential equation which may not be Fuchsian.
For example, Kummer's confluent hypergeometric function
\begin{equation}
\: _1 F _1 (\alpha ;\gamma ;z)=1+  \frac{\alpha  }{ \gamma } z + \frac{\alpha (\alpha +1) }{2! \: \gamma (\gamma +1)} z^2  + \dots + \frac{ (\alpha  )_n }{n!(\gamma )_n } z^n + \dots .
\end{equation}
satisfies the differential equation
\begin{equation}
 z \frac{d^2y}{dz^2} + \left( \gamma - z \right) \frac{dy}{dz} -\alpha  y=0,
\label{eq:KummerCHGE}
\end{equation}
which has an irregular singularity at $z=\infty $.
Eq.~(\ref{eq:KummerCHGE}) is called Kummer's differential equation or the confluent hypergeometric differential equation.

It is widely known that Kummer's function and Kummer's differential equation are obtained by confluence of the singularity of Gauss' hypergeometric equation.
Namely we replace the variable $z$ in Gauss' hypergeometric equation with $z/ \beta $ and take the limit $\beta \to \infty $.
Then the singularity $z=1 $ of Gauss' hypergeometric equation merges into the singularity $z=\infty $ and we obtain the confluent equation.

We can also consider the confluence process from Kummer's differential equation.
We set $z= u_1 x + u_2$, $u_2 = u_1 ^2 /2$, $\gamma = u_2$ and $\alpha =-\lambda /2$, and take the limit $u_1 \to \infty $.
Then we obtain the Hermite-Weber equation;
\begin{equation}
  \frac{d^2y}{dx ^2} -2x \frac{dy}{dx} + \lambda  y=0 .
\end{equation}
By setting $y= e^{x^2/2} u $, it is transformed to the differential equation $ -u'' +x^2 u= (1 + \lambda ) u $, which is related with the quantum harmonic oscillator.

The $q$-analogue of the hypergeometric functions has been studied well from the 19th century.
Heine's basic hypergeometric series was introduced as
\begin{equation}
_2 \phi _1 (a ,b ;c ;x) = \sum_{n=0}^{\infty} \frac{(a ;q)_n (b ;q)_n }{(q;q) _n (c ;q)_n } x^n, \quad (\lambda ; q)_n= \prod_{i=0}^{n-1}(1- \lambda q^i).
\label{eq:qhypser}
\end{equation}
It satisfies the equation
\begin{equation}
(x-q) f(x/q) - ((a+b)x -q-c)f(x)+ (abx-c)f(q x)=0. 
\label{eq:qhyp}
\end{equation}
It is known that Gauss' hypergeometric differential equation (\ref{eq:GaussHGE}) is obtained from Eq.~(\ref{eq:qhyp}) by the limit $q \to 1$, and the differential equation has singularities at $x=0,1,\infty $.

In this paper, we investigate a $q$-analogue of confluence processes which is related with the hypergeometric equations.
In particular, we consider the confluence processes from the variant of $q$-hypergeometric equation of degree two, which was introduced in \cite{HMST} as  
\begin{align}
& (x-q^{h_1 +1/2} t_1) (x - q^{h_2 +1/2} t_2) g(x/q) \label{eq:qhypervar1} \\
& + q^{\alpha _1 +\alpha _2} (x - q^{l_1-1/2}t_1 ) (x - q^{l_2 -1/2} t_2) g(q x) \nonumber \\
& -[ (q^{\alpha _1} +q^{\alpha _2} ) x^2 +E x + p ( q^{1/2}+ q^{-1/2}) t_1 t_2 ] g(x) =0, \nonumber \\
& p= q^{(h_1 +h_2 + l_1 + l_2 +\alpha _1 +\alpha _2 )/2 } , \quad E= -p \{ (q^{- h_2 }+q^{-l_2 })t_1 + (q^{- h_1 }+ q^{- l_1 }) t_2 \} . \nonumber
\end{align}
By taking the limit $q\to 1$, we essentially obtain the second order Fuchsian differential equation with three singularities $\{ t_1, t_2 ,\infty \}$ (see \cite{HMST}).
Recall that Hahn \cite{Hahn} introduced a $q$-difference analogue of Heun's differential equation of the form
\begin{align}
& \{ a_2 x^2 +a_1 x+ a_0 \} g(x/q)  -\{ b_2 x^2 + b_1 x + b_0 \} g(x) \label{eq:qHeun}\\
& \qquad + \{ c_2 x^2 + c_1 x+ c_0 \} g(xq) =0 , \nonumber
\end{align}
with the condition $a_2 a_0 c_2 c_0 \neq 0$, and it was rediscovered in \cite{TakR} by considering degenerations of the Ruijsenaars-van Diejen system.
We call Eq.~(\ref{eq:qHeun}) the $q$-Heun equation.
Then Eq.~(\ref{eq:qhypervar1}) is a specialization of the $q$-Heun equation, and it is characterized by the condition that the difference of the exponents at the origin $x=0$ is one and the singularity $x=0$ is apparent (see \cite{TakqH,HMST}).

We consider the degeneration of the variant of $q$-hypergeometric equation of degree two such that the polynomial $q^{\alpha _1 +\alpha _2} (x - q^{l_1-1/2}t_1 ) (x - q^{l_2 -1/2} t_2)$ tends to a linear polynomial.
We take the limit $q^{\alpha_2} \to 0  $ formally in Eq.~(\ref{eq:qhypervar1}) with the condition that $\alpha_2 +l_2 = h_1 +h_2 -l_1-\alpha _1 +1 -2\lambda $ is fixed.
Then we have
\begin{align}
\label{C-Heun}
&q^{ h_1+h_2-l_1-2\lambda +1/2}t_2(q^{l_1- 1/2}t_1-x)g(qx) \\
& +(x-q^{h_1 + 1/2} t_1)(x-q^{h_2 + 1/2}t_2)g(x/q) \nonumber \\
&-[q^{\alpha _1}x^2 -q^{h_1+h_2-\lambda +1/2}(q^{-h_2}t_1+q^{-h_1}t_2+q^{-l_1}t_2)x \nonumber \\
& \quad +q^{h_1+h_2-\lambda }(q + 1) t_1t_2]g(x)=0 . \nonumber 
\end{align} 
We may regard the parameter $\lambda $ in Eq.~(\ref{C-Heun}) to be independent from the other parameters $h_1$, $h_2$, $l_1$, $\alpha _1 $, $t_1$ and $t_2$.
We call Eq.~(\ref{C-Heun}) a variant of the singly confluent $q$-hypergeometric equation or the variant of the confluent $q$-hypergeometric equation of type $(1,2)$.
By taking the limit $q\to 1$, we essentially obtain Kummer's confluent hypergeometric equation although the position of the regular singularity is deformed to $x=t_1$ (see section \ref{sec:limdifftial} for details).

We consider further degeneration.
We take the limit $q^{-l_1} \to 0 $ formally in Eq.~(\ref{C-Heun}).
Then we have the following equation;
\begin{align}
\label{WC-Heun}
& g(qx)+   q^{ 2\lambda +1} ( 1 - q^{-h_1 -1/2} t_1 ^{-1} x )(1 - q^{-h_2 -1/2 } t_2 ^{-1} x )g(x/q)\\
&-[q^{\alpha _1 + 2\lambda -h_1- h_2 } t_1 ^{-1} t_2 ^{-1} x^2 - q^{ \lambda +1/2}(q^{-h_1}t_1 ^{-1} + q^{-h_2}t_2 ^{-1})x +q^{ \lambda } (q + 1) ]g(x)=0 . \nonumber 
\end{align}
We call Eq.~(\ref{WC-Heun}) a variant of the biconfluent $q$-hypergeometric equation or the variant of the confluent $q$-hypergeometric equation of type $(0,2)$.
By taking the limit $q\to 1$, we essentially obtain the Hermite-Weber differential equation (see section \ref{sec:limdifftial}).

Next we investigate solutions of the $q$-difference equations.
It was discovered in \cite{HMST} that the variant of $q$-hypergeometric equation of degree two (Eq.~(\ref{eq:qhypervar1})) has several explicit formal solutions, and we describe them in Proposition \ref{thm:thmHMST}.
For example, the function 
\begin{align}
& g (x) = x^{\lambda } \sum _{n=0}^{\infty} c_n \Big( \frac{x}{q^{l_1-1/2} t_1} ;q \Big)_n , \; c_n= \frac{(q^{\lambda +\alpha _1 };q )_n (q^{\lambda +\alpha _2 };q )_n q^n }{(q^{h_1 - l_1 +1};q )_n (q^{h_{2} - l_1 +1} t_{2} /t_1 ;q )_n (q;q)_n } \label{eq:solex}
\end{align}
is a formal solution of Eq.~(\ref{eq:qhypervar1}), where $\lambda = (h_1 +h_2 -l_1-l_2 -\alpha _1-\alpha _2+1 )/2$.
Here the formal solution means that the coefficients of the solution (e.g.~$c_n$ in Eq.~(\ref{eq:solex})) are determined recursively.
On the degeneration to the variant of the singly confluent $q$-hypergeometric equation (Eq.~(\ref{C-Heun})) as $q^{\alpha_2} \to 0  $, we can also take the limit of the solutions of Eq.~(\ref{eq:qhypervar1}), and we obtain several explicit formal solutions of Eq.~(\ref{C-Heun}) (see theorems in section \ref{sec:SClimit}).
For example, the function
\begin{align}
& g (x) = x^{\lambda } \sum _{n=0}^{\infty} c_n \Big( \frac{x}{q^{l_1-1/2} t_1} ;q \Big)_n , \; c_n= \frac{(q^{\lambda +\alpha _1 };q )_n  q^n }{(q^{h_1 - l_1 +1};q )_n (q^{h_{2} - l_1 +1} t_{2} /t_1 ;q )_n (q;q)_n }  \label{eq:solconfex} 
\end{align}
is a solution of Eq.~(\ref{C-Heun}), in which the term $(q^{\lambda +\alpha _2 };q )_n $ in Eq.(\ref{eq:solex}) is replaced with $1$ by the limit $q^{\alpha_2} \to 0$.
We can also several explicit formal solutions of Eq.~(\ref{WC-Heun}) by the limit $q^{-l_1} \to 0 $ in Eq.~(\ref{C-Heun}) (see theorems in section \ref{sec:BClimit}).
On this research, we eventually found new solutions of Eq.~(\ref{eq:qhypervar1}), which we note in Theorem \ref{thm:newsol}.
We can also consider the limits of the functions in Theorem \ref{thm:newsol} and we obtain several explicit formal solutions of Eq.~(\ref{C-Heun}) and those of Eq.~(\ref{WC-Heun}).

Recall that the coefficient of $g(qx)$ in the variant of the singly confluent $q$-hypergeometric equation (i.e.~Eq.~(\ref{C-Heun})) is a linear polynomial.
We can also consider the degeneration that the coefficient of $g(x/q)$ tends to a linear polynomial.
We perform it in section \ref{sec:otherconfl} and we obtain another variant of the singly confluent $q$-hypergeometric equation (see Eq.~(\ref{C2-Heun})) together with several formal solutions of that.
We also obtain another variant of the biconfluent $q$-hypergeometric equation (see Eq.~(\ref{WC2-Heun})) and several formal solutions.
It is seen in section \ref{sec:gauge} that another variant of the singly confluent $q$-hypergeometric equation (resp.~another variant of the biconfluent $q$-hypergeometric equation) is transformed to Eq.~(\ref{C-Heun}) (resp.~Eq.~(\ref{WC-Heun})) by multiplying a scalar function (a gauge factor).
Then we obtain a solution of Eq.~(\ref{C-Heun}) (resp.~Eq.~(\ref{WC-Heun})) by multiplying a solution of another variant of the singly confluent $q$-hypergeometric equation (resp.~another variant of the biconfluent $q$-hypergeometric equation) by the gauge factor.
 
This paper is organized as follows.
In section \ref{sec:sol}, we review some solutions of the variant of $q$-hypergeometric equation of degree two (Eq.~(\ref{eq:qhypervar1})) and we obtain new solutions.
In section \ref{sec:SClimit}, we obtain some solutions of the variant of the singly confluent $q$-hypergeometric equation (Eq.~(\ref{C-Heun})).
In section \ref{sec:BClimit}, we obtain some solutions of the variant of the singly confluent $q$-hypergeometric equation (Eq.~(\ref{WC-Heun})).
In section \ref{sec:otherconfl}, we discuss other variants of the singly confluent $q$-hypergeometric equation and  the biconfluent $q$-hypergeometric equation.
In section \ref{sec:gauge}, we discuss relationships between another variant of the singly confluent $q$-hypergeometric equation (resp.~another variant of the biconfluent $q$-hypergeometric equation) and  Eq.~(\ref{C-Heun}) (resp.~Eq.~(\ref{WC-Heun})), and we give applications to the solutions. 
In section \ref{sec:limdifftial}, we consider the limits to the singly confluent differential equation of Kummer and the biconfluent differential equation of Hermite-Weber as $q \to 1$.
In section \ref{sec:concl}, we give concluding remarks.

\section{Solutions to the variant of $q$-hypergeometric equation of degree two} \label{sec:sol}

Recall that the variant of $q$-hypergeometric equation of degree two was given in Eq.~(\ref{eq:qhypervar1}), i.e.
\begin{align*}
& (x-q^{h_1 +1/2} t_1) (x - q^{h_2 +1/2} t_2) g(x/q)  + q^{\alpha _1 +\alpha _2} (x - q^{l_1-1/2}t_1 ) (x - q^{l_2 -1/2} t_2) g(q x) \\
&  \qquad -[ (q^{\alpha _1} +q^{\alpha _2} ) x^2 +E x + p ( q^{1/2}+ q^{-1/2}) t_1 t_2 ] g(x) =0, \nonumber \\
& p= q^{(h_1 +h_2 + l_1 + l_2 +\alpha _1 +\alpha _2 )/2 } , \quad E= -p \{ (q^{- h_2 }+q^{-l_2 })t_1 + (q^{- h_1 }+ q^{- l_1 }) t_2 \} . \nonumber
\end{align*}
In this paper we do not consider convergence of solutions of the $q$-difference equations, and we treat solutions formally.
If the infinite summation in the formal solution terminates as a finite summation, then it is the exact solution.
Several solutions of the variant of $q$-hypergeometric equation of degree two were obtained in \cite{HMST} as follows;
\begin{prop} $($\cite{HMST}$)$ \label{thm:thmHMST}
Let $\lambda = (h_1 +h_2 -l_1-l_2 -\alpha _1-\alpha _2+1 )/2$.\\
(i) The function 
\begin{align}
g (x) = & x^{-\alpha _1 } \sum _{n=0}^{\infty} (q^{1/2} x^{-1})^n \frac{ ( q^{\lambda +\alpha _1 } ; q )_n }{(q^{\alpha _1  -\alpha _2 +1 } ; q )_n } \label{eq:g1x} \\
& \cdot \sum _{k=0}^n  \frac{(q^{ \lambda +\alpha _1  -h_2 +l_2 }; q)_k (q^{  \lambda +\alpha _1  -h_1 +l_1 } ;q)_{n-k }}{(q;q )_k (q;q)_{n-k}}  (q^{l_1} t_1)^k (q ^{l_2}t_2 ) ^{n-k} \nonumber 
\end{align}
is a solution of the variant of $q$-hypergeometric equation of degree two (i.e.~Eq.~(\ref{eq:qhypervar1})).
(ii) Set $(i,i')=(1,2) $ or $(2,1)$.
Then the function
\begin{align}
 g (x) & = x^{\lambda } \sum _{n=0}^{\infty} \Big( \frac{x}{q^{l_i-1/2} t_i} ;q \Big)_n \frac{(q^{\lambda +\alpha _1 };q )_n (q^{\lambda +\alpha _2 };q )_n }{(q^{h_i - l_i +1};q )_n (q^{h_{i'} - l_i +1} t_{i'} /t_i ;q )_n (q;q)_n } q^n \\
& \biggl( = x^{\lambda }\  _3\phi_2\biggl(q^{\lambda + \alpha _1},q^{\lambda +\alpha _2 },\frac{x}{q^{l_i-1/2}t_i};q^{h_i-l_i+1},q^{h_{i'}-l_i +1}t_{i'}/t_i ;q , q \biggr) \biggr) \nonumber
\end{align}
is a solution of Eq.~(\ref{eq:qhypervar1}).\\
(iii) Set $(i,i')=(1,2) $ or $(2,1)$.
Then the function
\begin{align}
& g (x)= x^{-\alpha _1 } \Big[ \sum _{n=0}^{\infty} \Big( \frac{q^{h_i+1/2} t_i }{x}  ;q \Big)_n \frac{(q^{\lambda +\alpha _1 };q )_n }{(q^{h_i - l_{i'} +1} t_i/t_{i'} ;q )_n  } q^n \\
& \qquad \qquad \cdot \sum _{k=0}^n   \frac{(q^{\lambda -h_{i'} +l_{i'} +\alpha _1} ;q)_k }{(q^{h_i-l_i+1};q)_{k} (q;q)_k (q;q)_{n-k} } q^{k(k+1)/2} (-q^{h_i-l_{i'} }t_i/t_ {i'})^k  \Big] 
\nonumber
\end{align}
is a solution of Eq.~(\ref{eq:qhypervar1}).
\end{prop}
Note that the functions which are obtained by replacing $\alpha _1 $ with $ \alpha _2$ are also solutions of the variant of $q$-hypergeometric equation of degree two.

We discovered new solutions of Eq.~(\ref{eq:qhypervar1}) as follows, which can be confirmed similarly to Proposition \ref{thm:thmHMST} established in \cite{HMST}.
\begin{thm} \label{thm:newsol}
Let $\lambda = (h_1 +h_2 -l_1-l_2 -\alpha _1-\alpha _2+1 )/2$. Set $(i,i')=(1,2) $ or $(2,1)$.\\
(i) The function
\begin{align}
 g (x) & = x^{\lambda } \sum _{n=0}^{\infty} \frac{(q^{\lambda +\alpha _1 };q )_n (q^{\lambda +\alpha _2 };q )_n }{(q^{h_i - l_i +1};q )_n (q^{h_{i} - l_{i'} +1} t_{i} /t_{i'} ;q )_n (q;q)_n } \Big( \frac{ q^{h_i+1/2} t_i}{x} ;q \Big)_n \Big( \frac{x}{q^{h_{i'}-1/2} t_{i'}} \Big)^n \\
& \biggl( = x^{\lambda }\  _3\phi_2\biggl(q^{\lambda + \alpha_1},q^{\lambda +\alpha _2 },\frac{q^{h_i+1/2} t_i}{x} ; q^{h_i - l_i +1} ,q^{h_{i} - l_{i'} +1} t_{i} /t_{i'}; q ,\frac{x}{q^{h_{i'}-1/2} t_{i'}} \biggr) \biggr) \nonumber  
\end{align}
is a solution of Eq.~(\ref{eq:qhypervar1}).\\
(ii) The function
\begin{align}
& g (x)= x^{\alpha _1 } \sum _{n=0}^{\infty} c_n  \frac{(q^{\lambda + \alpha _1 } ;q)_n }{(q^{1+h_{i'} -l_{i}} t_{i'} /t_{i} ;q )_n  } \Big( \frac{x}{ q^{l_{i} -1/2} t_{i}} ;q \Big)_n \Big( \frac{q^{-\lambda - \alpha _1 +h_{i'} +1/2 } t_{i'} }{x} \Big)^n ,  \\
& c_n = \sum _{k =0}^n q^{-nk + k^2/2} \frac{( q^{ \lambda + \alpha _1 - h_{i'} +l_{i'} } ;q)_{k}  }{(q^{h_{i} -l_{i} +1};q)_{k } (q;q)_{n-k} (q;q)_{k} } \Big(- \frac{q^{ -\lambda -\alpha _1+ h_{i} +1 /2 } t_{i} }{q^{l_{i'}} t_{i'}} \Big)^{k} \nonumber
\end{align}
is a solution of Eq.~(\ref{eq:qhypervar1}).
\end{thm}
The functions which are obtained by replacing $\alpha _1 $ with $ \alpha _2$ are also solutions of the variant of $q$-hypergeometric equation of degree two.
Note that solutions in Theorem \ref{thm:newsol} correspond to the ones obtained by replacing the role of $q$ with $q^{-1}$ in Proposition \ref{thm:thmHMST}.

If $\lambda +\alpha _1$ is a negative integer and set $N= -\lambda -\alpha _1$, then the summation of each solution is finite, and the solution is a product of $x^{\lambda } $ and the polynomial of the variable $x$ of degree $N -1$.
The polynomial essentially coincides with the Big $q$-Jacobi polynomial (cf.~\cite{GR,KLS}).

\section{Singly confluent limit} \label{sec:SClimit}

In the introduction, we introduced a variant of the singly confluent $q$-hypergeometric equation or the variant of the confluent $q$-hypergeometric equation of type $(1,2)$ as Eq.~(\ref{C-Heun}), i.e.
\begin{align*}
&q^{ h_1+h_2-l_1-2\lambda +1/2}t_2(q^{l_1- 1/2}t_1-x)g(qx)+(x-q^{h_1 + 1/2} t_1)(x-q^{h_2 + 1/2}t_2)g(x/q)\\
&-[q^{\alpha _1}x^2 -q^{h_1+h_2-\lambda +1/2}(q^{-h_2}t_1+q^{-h_1}t_2+q^{-l_1}t_2)x \nonumber \\
& \qquad \qquad \qquad \qquad \qquad \qquad +q^{h_1+h_2-\lambda }(q + 1)t_1t_2]g(x)=0 . \nonumber 
\end{align*} 
By taking the limit $q\to 1$, we essentially obtain the second order linear differential equation with one regular singularity $x= t_1$ and one irregular singularity $x= \infty $ (see Eq.~(\ref{eq:Cdiffeqgt})).

Recall that Eq.~(\ref{C-Heun}) was obtained by the limit $q^{\alpha_2} \to 0  $ from the variant of $q$-hypergeometric equation of degree two (i.e.~Eq.~(\ref{eq:qhypervar1})). 
We can obtain solutions of a variant of singly confluent $q$-hypergeometric equation by considering the limit of the solutions of the variant of $q$-hypergeometric equation.
By taking the limit $q^{\alpha_2} \to 0  $ in Proposition \ref{thm:thmHMST}, we may obtain solutions of Eq.~(\ref{C-Heun}).
\begin{thm} \label{thm:scsol}
(i) The function 
\begin{align}
g(x) & = x^{-\alpha _1 } \sum _{n =0}^{\infty} ( q^{ - \lambda - \alpha _1 +h_1 + 1/2} t_1 x^{-1} )^n ( q^{\lambda + \alpha _1 } ; q )_{n } \label{eq:sols000} \\
& \qquad \qquad \qquad \cdot \sum _{\ell =0}^{n } \frac{(q^{  \lambda +\alpha _1  -h_1 +l_1 } ;q)_{\ell }}{(q;q )_{n- \ell } (q;q)_{\ell }} q^{-\ell ( 2 n - \ell -1)/2}  (- q^{ -\lambda -\alpha _1 - l_1 + h_2  }t_2 /t_1 ) ^{\ell } \nonumber 
\end{align}
is a solution of Eq.~(\ref{C-Heun}).\\
(ii) The function 
\begin{align}
g(x) & = x^{\lambda } \sum^{\infty}_{n=0} 
q^{n}\frac{(q^{\lambda +\alpha_1};q)_n }{(q^{h_1-l_1+1};q)_n(q^{h_2-l_1+1}t_2/t_1;q)_n(q;q)_n}
\left(\frac{x}{q^{l_1-1/2}t_1 };q \right)_n  \label{eq:solssimpconf01} \\
& \biggl( = x^{\lambda }\  _3\phi_2\biggl(q^{\lambda + \alpha_1},0,\frac{x}{q^{l_1-1/2}t_1};q^{h_1-l_1+1},q^{h_2-l_1+1}t_2/t_1;q , q \biggr) \biggr) \nonumber
\end{align}
is a solution of Eq.~(\ref{C-Heun}).\\
(iii) The function 
\begin{align}
g(x) & = x^{-\alpha _1 } \Big[ \sum _{n=0}^{\infty} \Big( \frac{q^{h_1+1/2} t_1 }{x}  ;q \Big)_n  (q^{\lambda +\alpha _1 };q )_n q^n 
\sum _{k=0}^n   \frac{q^{k^2 } ( q^{\lambda +\alpha _1 + h_1 -h_2}t_1/t_2)^k }{(q^{h_1-l_1+1};q)_{k} (q;q)_k (q;q)_{n-k} }  \Big] \label{eq:sols0001}
\end{align}
is a solution of Eq.~(\ref{C-Heun}).\\
(iv) The function 
\begin{align}
g(x) & = x^{-\alpha _1 } \Big[ \sum _{n=0}^{\infty} \Big( \frac{q^{h_2+1/2} t_2 }{x}  ;q \Big)_n \frac{q^n (q^{\lambda +\alpha _1 };q )_n }{(q^{h_2 - l_1 +1} t_2/t_1 ;q )_n  }  \label{eq:sols0002} \\
& \qquad \qquad \qquad \cdot 
\sum _{k=0}^n   \frac{(q^{\lambda -h_1+l_1+\alpha _1} ;q)_k }{ (q;q)_k (q;q)_{n-k} } q^{k(k+1)/2} (-q^{h_2-l_1}t_2/t_1)^k  \Big] \nonumber 
\end{align}
is a solution of Eq.~(\ref{C-Heun}).
\end{thm}
\begin{proof}
We show (ii).
It was shown in \cite{HMST} that the function 
\begin{align}
g(x) & = x^{-\alpha _1 } \sum _{n =0}^{\infty} a_n \left(\frac{x}{q^{l_1-1/2}t_1 };q \right)_n \label{eq:gxan}
\end{align}
is a solution of Eq.~(\ref{eq:qhypervar1}), if the coefficient $a_n$ satisfies 
\begin{align}
& (1-q^{h_1-l_1+n+1})(1-q^{h_2-l_1+n+1} t_2/t_1)(1-q^{n+1})a_{n+1} \\
& -q (1-q^{\lambda +\alpha_1+n})(1-q^{\lambda +\alpha_2 + n}) a_n \nonumber \\
& - q^{3}(1+q^{-1})(1-q^{h_1-l_1+n})(1-q^{h_2-l_1+n}t_2/t_1)(1-q^{n}) a_{n} \nonumber \\
&  +q^{4}(1+q^{-1})(1-q^{\lambda +\alpha_1+n-1})(1-q^{\lambda +\alpha_2+n-1})a_{n-1} \nonumber \\
&  +q^{5}(1-q^{h_1-l_1+n-1})(1- q^{h_2-l_1+n-1} t_2/t_1 )(1-q^{n-1})a_{n-1} \nonumber \\
&  -q^{6}(1-q^{\lambda +\alpha_1+n-2})(1-q^{\lambda +\alpha_2+n-2})a_{n-2} =0 \nonumber 
\end{align}
for all integer $n$, where $a_n =0$ for $n <0$, and the sequence
\begin{align}
& a_n = q^{n}\frac{(q^{\lambda +\alpha_1};q)_n (q^{\lambda +\alpha_2};q)_n }{(q^{h_1-l_1+1};q)_n(q^{h_2-l_1+1}t_2/t_1;q)_n(q;q)_n} \quad (n\geq 0)
\end{align}
satisfies the recursive relation.
Thus we recover Proposition \ref{thm:thmHMST} (ii) for the case $(i,i')=(1,2)$.
By taking the limit $q^{\alpha_2} \to 0 $, the function $g(x)$ of the form in Eq.~(\ref{eq:gxan}) is a solution of Eq.~(\ref{C-Heun}), if the coefficient $a_n$ satisfies 
\begin{align}
& (1-q^{h_1-l_1+n+1})(1-q^{h_2-l_1+n+1} t_2/t_1)(1-q^{n+1})a_{n+1} \\
& -q (1-q^{\lambda +\alpha_1+n}) a_n - q^{3}(1+q^{-1})(1-q^{h_1-l_1+n})(1-q^{h_2-l_1+n}t_2/t_1)(1-q^{n}) a_{n} \nonumber \\
&  +q^{4}(1+q^{-1})(1-q^{\lambda +\alpha_1+n-1})a_{n-1} \nonumber \\
&   +q^{5}(1-q^{h_1-l_1+n-1})(1- q^{h_2-l_1+n-1} t_2/t_1 )(1-q^{n-1})a_{n-1} \nonumber \\
&  -q^{6}(1-q^{\lambda +\alpha_1+n-2})a_{n-2} =0 \nonumber 
\end{align}
for all integer $n$, where $a_n =0$ for $n <0$, and it is shown by taking the limit $q^{\alpha_2} \to 0 $ that the sequence 
\begin{align}
& a_n = q^{n}\frac{(q^{\lambda +\alpha_1};q)_n }{(q^{h_1-l_1+1};q)_n(q^{h_2-l_1+1}t_2/t_1;q)_n(q;q)_n} \quad (n\geq 0)
\end{align}
satisfies the recursive relation.
Namely we obtained (ii) by taking the limit of the function in Proposition \ref{thm:thmHMST} (ii) for the case $(i,i')=(1,2)$ as $q^{\alpha_2} \to 0 $. 

We can obtain Theorem (i) by taking the limit of the function in Proposition \ref{thm:thmHMST} (i) as $q^{\alpha_2} \to 0 $.
We can also obtain (iii) (resp.~(iv)) by taking the limit of the function in Proposition \ref{thm:thmHMST} (iii) for the case $(i,i')=(1,2)$ (resp.~for the case $(i,i')=(2,1) $) as $q^{\alpha_2} \to 0 $.
\end{proof}
By considering the limit of the functions in Theorem \ref{thm:newsol} (i) and Theorem \ref{thm:newsol} (ii) for the case $(i,i')=(1,2) $, we may guess the solutions of Eq.~(\ref{C-Heun}) and we can actually establish the following theorem.
\begin{thm} \label{thm:scsol2}
(i) Set $(i,i')=(1,2) $ or $(2,1)$. Then the function
\begin{align}
 g (x) & = x^{\lambda } \sum _{n=0}^{\infty}  \frac{(q^{\lambda +\alpha _1 };q )_n  }{(q^{h_i - l_1 +1} t_i /t_1 ;q )_n  (q;q)_n } \Big( \frac{ q^{h_i+1/2} t_i}{x} ;q \Big)_n \Big( \frac{x}{q^{h_{i'}-1/2} t_{i'}} \Big)^n \label{eq:sols0003} \\
& \biggl( = x^{\lambda }\  _2 \phi_1 \biggl(q^{\lambda + \alpha_1}, \frac{ q^{h_i+1/2} t_i}{x} ; q^{h_i - l_1 +1} t_i /t_1 ; q, \frac{x}{q^{h_{i'}-1/2} t_{i'}}  \biggr) \biggr) \nonumber
\end{align}
is a solution of Eq.~(\ref{C-Heun}).\\
(ii) The function
\begin{align}
& g (x)= x^{-\alpha _1 } \sum _{n=0}^{\infty} c_n \frac{(q^{\lambda + \alpha _1 } ;q)_n }{(q^{1+h_2 -l_1} t_2/t_1 ;q )_n  } \Big( \frac{x}{ q^{l_1 -1/2} t_1} ;q \Big)_n \Big( \frac{q^{-\lambda - \alpha _1 +h_2 +1/2 } t_2}{x} \Big)^n , \\
& c_n =  \sum _{k =0}^n q^{-nk + k^2} \frac{ 1 }{(q^{h_1-l_1 +1};q)_{k } (q;q)_{n-k} (q;q)_{k} } ( q^{ h_1 -h_2 } t_1 /  t_2 )^{k} \nonumber 
\end{align}
is a solution of Eq.~(\ref{C-Heun}).
\end{thm}
Note that, by the limit of the function in Theorem \ref{thm:newsol} (ii) for the case $(i,i')=(2,1) $ as $q^{\alpha_2} \to 0 $, we obtain Eq.~(\ref{eq:sols000}).

If $\lambda +\alpha _1$ is a negative integer and set $N= -\lambda -\alpha _1$, then the summation of each solution is finite, and the solution is a product of $x^{\lambda } $ and the polynomial of the variable $x$ of degree $N -1$.
The polynomial essentially coincides with the Big $q$-Laguerre polynomial (cf.~\cite{KLS}).

\section{Biconfluent limit} \label{sec:BClimit}
In the introduction, we introduced a variant of the biconfluent $q$-hypergeometric equation or the variant of the confluent $q$-hypergeometric equation of type $(0,2)$ as Eq.~(\ref{WC-Heun}), i.e.
\begin{align*}
& g(qx)+   q^{ 2\lambda +1} ( 1 - q^{-h_1 -1/2} t_1 ^{-1} x )(1 - q^{-h_2 -1/2 } t_2 ^{-1} x )g(x/q)\\
&-[q^{\alpha _1 + 2\lambda -h_1- h_2 } t_1 ^{-1} t_2 ^{-1} x^2 - q^{ \lambda +1/2}(q^{-h_1}t_1 ^{-1} + q^{-h_2}t_2 ^{-1})x +q^{ \lambda } (q+ 1) ]g(x)=0 . \nonumber 
\end{align*}
The equation has the symmetry of replacing $(t_1, h_1)$ with $(t_2 , h_2)$. 
By taking the limit $q\to 1$, we obtain the differential equation in Eq.~(\ref{eq:WCdiffeqgt}).

Eq.~(\ref{WC-Heun}) was obtained by the limit $q^{-l_1} \to 0 $ from a variant of singly confluent $q$-hypergeometric equation (i.e.~Eq.~(\ref{C-Heun})). 
We can obtain solutions of the variant of biconfluent $q$-hypergeometric equation by considering the limit of the solutions of the variant of singly confluent $q$-hypergeometric equation.
\begin{thm} \label{thm:bcsol}
(i) The function
\begin{align}
& x^{-\alpha _1 } \sum _{n =0}^{\infty} ( q^{ - \lambda - \alpha _1 + 1/2} x^{-1} )^{n } ( q^{\lambda + \alpha _1 } ; q )_{n } \sum _{\ell =0}^{n }  \frac{q^{- \ell (n -\ell )} ( q^{ h_1 } t_1 )^{n -\ell } (q^{ h_2 } t_2 ) ^{\ell } }{(q;q )_{n -\ell  } (q;q)_{\ell }} \label{eq:solbiconf00}
\end{align}
is a solution of Eq.~(\ref{WC-Heun}).\\
(ii) Set $(i,i')=(1,2) $ or $(2,1)$. Then the function
\begin{align}
& x^{-\alpha _1 } \Big[ \sum _{n=0}^{\infty} \Big( \frac{q^{h_i+1/2} t_i }{x}  ;q \Big)_n  (q^{\lambda +\alpha _1 };q )_n q^n 
\sum _{k=0}^n   \frac{q^{k^2 } ( q^{\lambda +\alpha _1 + h_i -h_{i'} }t_i/t_{i'} )^k }{ (q;q)_k (q;q)_{n-k} }  \Big] 
\end{align}
is a solution of Eq.~(\ref{WC-Heun}).\\
(iii) Set $(i,i')=(1,2) $ or $(2,1)$. Then the function
\begin{align}
 g (x) & = x^{\lambda } \sum _{n=0}^{\infty} \Big( \frac{q^{h_i+1/2} t_i }{x} ;q \Big)_n \frac{(q^{\lambda +\alpha _1 };q )_n  }{ (q;q)_n } \Big( \frac{x}{q^{h_{i'} -1/2} t_{i'}} \Big)^n  \label{eq:sols0004} \\
& \biggl( = x^{\lambda }\  _2 \phi_1 \biggl( q^{\lambda + \alpha_1} , \frac{q^{h_i+1/2} t_i }{x} ;0 ;q , \frac{x}{q^{h_{i'} -1/2} t_{i'}} \biggr) \biggr) \nonumber
\end{align}
is a solution of Eq.~(\ref{WC-Heun}).
\end{thm}
Note that the function in Theorem \ref{thm:bcsol} (i) (resp.~(ii), (iii)) is obtained by taking the limit of the function in Eq.~(\ref{eq:sols000}) (resp. Eq.~(\ref{eq:sols0001}) or Eq.~(\ref{eq:sols0002}), Eq.~(\ref{eq:sols0003})) as $q^{-l_1} \to 0 $.

If $\lambda +\alpha _1$ is a negative integer and set $N= -\lambda -\alpha _1$, then the summation of each solution is finite, and the solution is a product of $x^{\lambda } $ and the polynomial of the variable $x$ of degree $N -1$.
The polynomial essentially coincides with the Al-Salam-Carlitz I polynomial (cf.~\cite{KLS}).

\section{Other confluences} \label{sec:otherconfl}
We consider other confluence processes such that the degree of the polynomial on the coefficient of $g(x/q)$ decreases.

We take the limit $q^{-\alpha_2} \to 0  $ in Eq.~(\ref{eq:qhypervar1}) with the condition that the value $h_2 -\alpha_2  = - h_1 + l_1 + l_2 + \alpha _1 -1 + 2\lambda $ is fixed.
Then we have
\begin{align}
\label{C2-Heun}
& (x-q^{l_1-1/2}t_1)(x-q^{l_2-1/2}t_2)g(qx) \\
& - q^{- h_1 + l_1 + l_2 + 2\lambda  -1/2}t_2 (x-q^{h_1 +1/2} t_1) g(x/q) \nonumber \\
&-[\ q^{ - \alpha _1} x^2 - q^{ l_1 + l_2 + \lambda -1/2} \left\{ q^{-l_2} t_1+\left(q^{-h_1}+q^{-l_1}\right)t_2\right\} x \nonumber \\
& \quad + q^{l_1 + l_2 + \lambda } (1 + q^{-1} )t_1 t_2 ] g(x) =0 . \nonumber 
\end{align} 
We may regard the parameter $\lambda $ to be independent from the other parameters $h_1$, $h_2$, $l_1$, $\alpha _1 $, $t_1$ and $t_2$.
We call it another variant of the singly confluent $q$-hypergeometric equation or the variant of the confluent $q$-hypergeometric equation of type $(2,1)$.

We can obtain solutions of the variant of the confluent $q$-hypergeometric equation of type $(2,1)$ by considering the limit of the solutions of the variant of $q$-hypergeometric equation.
\begin{thm} \label{thm:oscsol}
(i) Set $(i,i')=(1,2) $ or $(2,1)$. Then the function
\begin{align}
& x^{\lambda } \sum^{\infty}_{n=0}  ( q^{- \lambda - \alpha _1 + h_1 - l_{i'} +1} t_1/t_{i'} )^n \frac{(q^{\lambda +\alpha_1};q)_n }{(q^{h_1-l_i +1}t_1 /t_i ;q)_n (q;q)_n} \left(\frac{x}{q^{l_i -1/2} t_i };q \right)_n \label{eq:sols000a} \\
& \biggl( = x^{\lambda }\  _2 \phi_1 \biggl(q^{\lambda +\alpha_1} , \frac{x}{q^{l_i-1/2}t_i };q^{h_1-l_i +1}t_1 /t_i ; q, q^{ - \lambda - \alpha _1 + h_1 - l_{i'} +1} t_1/t_{i'} \biggr) \biggr) \nonumber
\end{align}
is a solution of Eq.~(\ref{C2-Heun}).\\
(ii) The function 
\begin{align}
g (x) = & x^{-\alpha _1 } \sum _{n=0}^{\infty} ( q^{1/2} x^{-1})^n ( q^{\lambda +\alpha _1 } ; q )_n \sum _{k=0}^n  \frac{ (q^{  \lambda +\alpha _1  -h_1 +l_1 } ;q)_{n-k }}{(q;q )_k (q;q)_{n-k}}  (q^{l_1 } t_1)^k (q ^{l_2 } t_2 ) ^{n-k} \label{eq:sols000b}
\end{align}
is a solution of Eq.~(\ref{C2-Heun}).\\
(iii) The function 
\begin{align}
& g (x)= x^{-\alpha _1 } \Big[ \sum _{n=0}^{\infty} ( q^{h_1 +1/2} t_1 /x  ;q )_n \frac{(q^{\lambda +\alpha _1 };q )_n }{(q^{h_1 - l_{2} +1} t_1/t_{2} ;q )_n  } q^n \sum _{k=0}^n   \frac{  q^{k(k+1)/2} (-q^{h_1 -l_{2} } t_1 /t_ {2})^k    }{(q^{h_1 -l_1+1};q)_{k} (q;q)_k (q;q)_{n-k} } \Big] \label{eq:sols000c}
\end{align}
is a solution of Eq.~(\ref{C2-Heun}).
\end{thm}
We can obtain the functions in Theorem \ref{thm:oscsol} (i) (resp. (ii), (iii)) by taking the limit of the function in Proposition \ref{thm:thmHMST} (ii) (resp.  Proposition \ref{thm:thmHMST} (i), Proposition \ref{thm:thmHMST} (iii) for the case $(i,i')=(1,2)$) as $q^{-\alpha_2} \to 0 $. 
We can also obtain the functions in Theorem \ref{thm:oscsol} (ii) by taking the limit of the function in Proposition \ref{thm:thmHMST} (iii) for the case $(i,i')=(2,1)$.
\begin{thm} \label{thm:oscsol2}
(i) The function
\begin{align}
 g (x) & = x^{\lambda } \sum _{n=0}^{\infty} \Big( \frac{ q^{h_1+1/2} t_1}{x} ;q \Big)_n \frac{(q^{\lambda +\alpha _1 };q )_n q^{n(n-1)/2} ( -q^{ -\lambda - \alpha _1 - l_1 - l_2 + h_1 +3/2 } x/ t_{2} )^n }{(q^{h_1 - l_1 +1};q )_n (q^{h_{1} - l_{2} +1} t_{1} /t_{2} ;q )_n (q;q)_n } \label{eq:sols000d}
\\
& \biggl( = x^{\lambda }\  _2 \phi_2 \biggl(q^{\lambda +\alpha_1} , \frac{ q^{h_1+1/2} t_1}{x} ;q^{h_1 - l_1 +1} , q^{h_{1} - l_{2} +1} t_{1} /t_{2} ; q, q^{ -\lambda - \alpha _1 - l_1 - l_2 + h_1 +3/2 } x / t_{2} \biggr) \biggr) \nonumber
\end{align}
is a solution of Eq.~(\ref{C2-Heun}).\\
(ii) The function
\begin{align}
& g (x)= x^{-\alpha _1 } \sum _{n=0}^{\infty} c_n \Big( \frac{x}{ q^{l_1 -1/2} t_1} ;q \Big)_n \Big( - \frac{q^{-\lambda - \alpha _1 +l_1 } t_1}{x} \Big)^n (q^{\lambda + \alpha _1 } ;q)_n q^{-n^2/2 } , \label{eq:sols000e} \\
& c_n =  \sum _{k =0}^n q^{-nk + k^2/2} \frac{1 }{(q^{h_1-l_1 +1};q)_{k } (q;q)_{n-k} (q;q)_{k} } \Big(- \frac{q^{ -\lambda -\alpha _1+ h_1 +1 /2 } t_1 }{q^{l_2} t_2} \Big)^{k} \nonumber 
\end{align}
is a solution of Eq.~(\ref{C2-Heun}).\\
(iii) The function
\begin{align}
& g (x)= x^{-\alpha _1 } \sum _{n=0}^{\infty} c_n \Big( \frac{x}{ q^{l_2 -1/2} t_2} ;q \Big)_n \Big( \frac{q^{-\lambda - \alpha _1 +h_1 +1/2 } t_1}{x} \Big)^n \frac{(q^{\lambda + \alpha _1 } ;q)_n }{(q^{1+h_1 -l_2} t_1/t_2 ;q )_n  } , \label{eq:sols000f} \\
& c_n =   \sum _{k =0}^n q^{-nk } \frac{( q^{ \lambda + \alpha _1 - h_1 +l_1 } ;q)_{k}  }{(q;q)_{n-k} (q;q)_{k} } \Big( \frac{q^{ -\lambda -\alpha _1+ l_2 } t_2 }{q^{l_1} t_1} \Big)^{k} \nonumber
\end{align}
is a solution of Eq.~(\ref{C2-Heun}).
\end{thm}
The function in Theorem \ref{thm:oscsol2} (i) (resp. (ii), (iii)) appears as the limit of the function in Theorem \ref{thm:newsol} (i) for the case $(i,i')=(1,2) $ (resp. Theorem \ref{thm:newsol} (ii) for the case $(i,i')=(1,2) $, Theorem \ref{thm:newsol} (ii) for the case $(i,i')=(2,1) $) as $q^{-\alpha_2} \to 0 $.

To obtain a variant of another biconfluent $q$-hypergeometric equation, we take the limit $ q^{-h_1} \to 0 $ in Eq.~(\ref{C2-Heun}).
Then we have
\begin{align}
& q^{ - 2\lambda -1}  (1 - q^{-l_1+ 1/2}  t_1^{-1} x)( 1 -q^{-l_2 + 1/2} t_2^{-1} x )g(qx) +  g(x/q) \label{WC2-Heun} \\
& -[ q^{ - \alpha _1 - l_1 - l_2 - 2\lambda } t_1^{-1} t_2^{-1} x^2 - q^{ - \lambda -1/2 } ( q^{-l_1} t_1 ^{-1} + q^{-l_2} t_2 ^{-1} ) x + q^{- \lambda } (1+q^{-1}) ] g(x) =0 . \nonumber 
\end{align} 
We call it another variant of the biconfluent $q$-hypergeometric equation or the variant of the confluent $q$-hypergeometric equation of type $(2,0)$.
The equation has the symmetry of replacing $(t_1, l_1)$ with $(t_2 , l_2)$. 
By taking the limit $q\to 1$, we essentially obtain the Hermite-Weber differential equation (see section \ref{sec:limdifftial}).

We can obtain solutions of the variant of the confluent $q$-hypergeometric equation of type $(2,0)$ by considering the limit of the solutions of the variant of the confluent $q$-hypergeometric equation of type $(2,1)$.
\begin{thm} \label{thm:obcsol2}
(i) Set $(i,i')=(1,2) $ or $(2,1)$. Then the function
\begin{align}
& x^{\lambda } \sum^{\infty}_{n=0}  ( - q^{- \lambda - \alpha _1 + l_i - l_{i'} } t_i/t_{i'} )^n \frac{q^{-n(n-1)/2}(q^{\lambda +\alpha_1};q)_n }{(q;q)_n} \left(\frac{x}{q^{l_i -1/2} t_i };q \right)_n \label{eq:sols000x} \\
& \biggl( =  x^{\lambda }\  _2 \phi_0 \biggl(q^{\lambda +\alpha_1} , \frac{x}{q^{l_i-1/2}t_i }; - ; q, q^{ - \lambda - \alpha _1 + l_i - l_{i'} } t_i/t_{i'} \biggr) \biggr) \nonumber
\end{align}
is a solution of Eq.~(\ref{WC2-Heun}).\\
(ii) The function 
\begin{align}
g (x) = & x^{-\alpha _1 } \sum _{n=0}^{\infty} (q^{1/2} x^{-1})^n ( q^{\lambda +\alpha _1 } ; q )_n \sum _{k=0}^n  \frac{ (q^{l_1 } t_1)^k (q ^{l_2 }t_2 ) ^{n-k} }{(q;q )_k (q;q)_{n-k}} 
\end{align}
is a solution of Eq.~(\ref{WC2-Heun}).\\
(iii) Set $(i,i')=(1,2) $ or $(2,1)$. Then the function
\begin{align}
& g (x)= x^{-\alpha _1 } \sum _{n=0}^{\infty} c_n \Big( \frac{x}{ q^{l_i -1/2} t_i} ;q \Big)_n \Big( - \frac{q^{-\lambda - \alpha _1 +l_i } t_i}{x} \Big)^n (q^{\lambda + \alpha _1 } ;q)_n q^{-n^2/2 } ,  \\
& c_n =  \sum _{k =0}^n q^{-nk } \frac{1 }{ (q;q)_{n-k} (q;q)_{k} } \Big( \frac{q^{ -\lambda -\alpha _1+ l_i -l_{i'} } t_i }{ t_{i'}} \Big)^{k} \nonumber 
\end{align}
is a solution of Eq.~(\ref{WC2-Heun}).
\end{thm}
We can obtain the function in Theorem \ref{thm:obcsol2} (i) (resp.~Theorem \ref{thm:obcsol2} (ii)) by taking the limit of the function in Eq.~(\ref{eq:sols000a}) (resp.~Eq.~(\ref{eq:sols000b})) as $q^{-h_1} \to 0$. 
We can also obtain the function in Theorem \ref{thm:obcsol2} (iii) for the case $(i.i')=(1,2)$ (resp. the case $(i.i')=(2,1)$) by taking the limit of the function in Eq.~(\ref{eq:sols000e}) (resp.~Eq.~(\ref{eq:sols000f})) as $q^{-h_1} \to 0$.

\section{Gauge transformation} \label{sec:gauge} 

We investigate the gauge transformation of linear difference equations.
We apply it to the correspondence between the variant of the confluent $q$-hypergeometric equation of type $(i,j)$ and that of type $(j,i)$.
We set $(\alpha x;q) _{\infty } = \prod _{i=0}^{\infty } (1- q^i \alpha x) $.
\begin{prop} \label{prop:GT}
(i) If $y(x)$ is a solution of the difference equation 
\begin{equation}
(1- \alpha x ) a(x) g(x/q) + b(x) g(x) +  c(x) g(qx) =0,
\label{eq:qDeq1-1}
\end{equation}
then the function $u (x)= (\alpha q x ;q) _{\infty } y(x) $ satisfies 
\begin{equation}
 a(x) g(x/q) + b(x) g(x) + (1- \alpha q x ) c(x) g(qx) =0 .
\end{equation}
(ii) If $y(x)$ is a solution of the difference equation 
\begin{equation}
a(x) g(x/q) + b(x) g(x) + (1- \alpha x ) c(x) g(qx) =0,
\end{equation}
then the function $u (x)= y(x) / (\alpha x;q) _{\infty } $ satisfies 
\begin{equation}
(1-\alpha x/q) a(x) g(x/q) + b(x) g(x) + c(x) g(qx) =0 .
\end{equation}
\end{prop}
\begin{proof}
It follows from the definition that $(\alpha x;q) _{\infty } = (1-\alpha x ) (\alpha q x;q) _{\infty }$ and $(\alpha q^2 x ;q) _{\infty }= (\alpha q x;q) _{\infty } / (1-\alpha q x) $.
If the function $y(x)$ satisfies Eq.~(\ref{eq:qDeq1-1}) and $ u(x) =  (\alpha q x;q) _{\infty } y(x )$, then we have 
\begin{align}
& a(x) u(x/q) + b(x) u(x) + (1- \alpha q x ) c(x) u(qx) \\
& = (\alpha q x;q) _{\infty } \{  (1- \alpha x )  a(x) y(x/q) + b(x) y(x) + c(x) y(qx) \}  =0 . \nonumber
\end{align}
Hence we obtain (i). 
We can show (ii) similarly.
\end{proof}
By applying \ref{prop:GT}, we obtain correspondences among the variants of the confluent $q$-hypergeometric equation.
A correspondence between Eq.~(\ref{C-Heun}) and Eq.~(\ref{C2-Heun}) is given as follows.
\begin{prop} \label{prop:CC2corresp}
Assume that function $y(x)$ satisfies Eq.~(\ref{C-Heun}), which is a variant of the confluent $q$-hypergeometric equation of type $(1,2)$. 
Then the function $u(x) = (q^{-h_2 +1/2} t_2^{-1} x;q)_{\infty } y(x ) $ is a solution of the equation
\begin{align}
& (x-q^{\tilde{l}_1-1/2}t_1)(x-q^{\tilde{l}_2-1/2}t_2)g(qx) \label{C2-Heun01} \\
& - q^{- \tilde{h}_1 + \tilde{l}_1 + \tilde{l}_2 + 2\lambda  -1/2}t_2 (x-q^{\tilde{h}_1 +1/2} t_1) g(x/q) \nonumber \\
&-[\ q^{ - \tilde{\alpha }_1} x^2 - q^{ \tilde{l}_1 + \tilde{l}_2 + \lambda -1/2} \{ q^{-\tilde{l}_2} t_1+ ( q^{-\tilde{h}_1}+q^{-\tilde{l}_1} ) t_2 \}  x \nonumber \\
& \quad + q^{\tilde{l}_1 + \tilde{l}_2 + \lambda } (q+1)t_1 t_2 ] g(x) =0 , \nonumber  
\end{align}
where
\begin{align}
& \tilde{l}_1= l_1,\; \tilde{l}_2 =h_2, \; \tilde{h}_1 =h_1,\; \tilde{\alpha }_1 + \alpha _1 = h_1 -l_1 + 1 - 2\lambda .
\label{eq:propCparam}
\end{align}
Note that Eq.~(\ref{C2-Heun01}) is a variant of the confluent $q$-hypergeometric equation of type $(2,1)$ (see Eq.~(\ref{C2-Heun})).
Conversely, if the function $u(x) $ satisfies Eq.~(\ref{C2-Heun01}), then the function $ y(x) = u(x) / (q^{-\tilde{l}_2 +1/2} t_2^{-1} x;q)_{\infty } $ satisfies Eq.~(\ref{C-Heun}) where the relationship among the parameters is given by Eq.~(\ref{eq:propCparam}).
\end{prop}
\begin{proof}
It follows from Eq.~(\ref{C-Heun}) that
\begin{align*}
& (1 -q^{-l_1+ 1/2} t_1^{-1} x)g(qx)+ q^{2\lambda + 1} (1- q^{-h_1 - 1/2} t_1^{-1} x)(1-  q^{-h_2 - 1/2}t_2^{-1} x )g(x/q)\\
& -[ q^{2\lambda + \alpha _1 -h_1 - h_2 } t_1^{-1} t_2^{-1} x^2 - q^{\lambda +1/2}(q^{-h_2} t_2^{-1} +q^{-h_1} t_1^{-1} +q^{-l_1} t_1^{-1} )x \nonumber \\
& \qquad + q^{ \lambda }(q+ 1) ]g(x)=0 . \nonumber 
\end{align*}
We apply Proposition \ref{prop:GT} (i).
Set $\alpha = q^{-h_2 - 1/2} t_1^{-1}$. Then the function $u (x)= (q^{-h_2 +1/2} t_2^{-1} x ;q) _{\infty } y(x) $ satisfies
\begin{align*}
&  (1 -q^{-l_1+ 1/2} t_1^{-1} x) (1-  q^{-h_2 + 1/2}t_2^{-1} x ) g(qx)+ q^{2\lambda + 1} (1- q^{-h_1 - 1/2} t_1^{-1} x)g(x/q)\\
&-[ q^{2\lambda + \alpha _1 -h_1 - h_2 } t_1^{-1} t_2^{-1} x^2 - q^{\lambda +1/2}(q^{-h_2} t_2^{-1} +q^{-h_1} t_1^{-1} +q^{-l_1} t_1^{-1} )x \nonumber \\
& \qquad + q^{ \lambda }(q+1) ]g(x)=0 . \nonumber
\end{align*}
On the other hand, Eq.~(\ref{C2-Heun01}) is rewritten as 
\begin{align*}
&  ( 1 -q^{-\tilde{l}_1+1/2}t_1^{-1} x )( 1 -q^{-\tilde{l}_2+1/2}t_2^{-1} x )g(qx) + q^{2\lambda +1 } (1- q^{- \tilde{h}_1 -1/2} t_1^{-1} x ) g(x/q) \\
&-[ q^{ - \tilde{\alpha }_1 -\tilde{l}_1 - \tilde{l}_2 +1 } t_1^{-1} t_2^{-1} x^2 - q^{ \lambda +1/2} ( q^{-\tilde{l}_2} t_2^{-1} + q^{-\tilde{h}_1} t_1^{-1} + q^{-\tilde{l}_1} t_1^{-1} ) x \nonumber \\
& \qquad + q^{ \lambda } (q+1) ] g(x) =0 . \nonumber 
\end{align*}
By these expressions, we obtain the proposition.
\end{proof}
As a consequence of Proposition \ref{prop:CC2corresp}, we obtain the following series solutions at $x=\infty $.
\begin{prop}
(i) The functions
\begin{align}
& g_1 (x)= x^{-\alpha _1 } \sum _{n =0}^{\infty} ( q^{ - \lambda - \alpha _1 +h_1 + 1/2} t_1 x^{-1} )^n ( q^{\lambda + \alpha _1 } ; q )_{n } \\
& \qquad \qquad \qquad \cdot \sum _{\ell =0}^{n } \frac{(q^{  \lambda +\alpha _1  -h_1 +l_1 } ;q)_{\ell }}{(q;q )_{n- \ell } (q;q)_{\ell }} q^{-\ell ( 2 n - \ell -1)/2}  (- q^{ -\lambda -\alpha _1 - l_1 + h_2  }t_2 /t_1 ) ^{\ell } , \nonumber \\
 & g_2(x) = ((q^{-h_2 +1/2} t_2^{-1} x;q)_{\infty } )^{-1} x^{ 2\lambda  + \alpha _1 - h_1 + l_1 - 1 } \sum _{n=0}^{\infty} ( q^{1/2} x^{-1})^n  \nonumber \\
& \qquad \qquad \cdot ( q^{-\lambda  -\alpha _1 + h_1 -l_1 + 1 } ; q )_n \sum _{k=0}^n  \frac{ (q^{ -\lambda  -\alpha _1 + 1 } ;q)_{n-k }}{(q;q )_k (q;q)_{n-k}}  (q^{l_1 } t_1)^k (q ^{h_2 } t_2 ) ^{n-k} \nonumber
\end{align}
are solutions of Eq.~(\ref{C-Heun}).\\
(ii)
The functions
\begin{align}
& g_3(x) = (q^{-l_2 +1/2} t_2^{-1} x;q)_{\infty } x^{ 2\lambda + \alpha _1 - h_1 +l_1 - 1 } \sum _{n =0}^{\infty} ( q^{ \lambda + \alpha _1 + l_1 - 1/2 } t_1 x^{-1} )^n \\
& \qquad \cdot ( q^{-\lambda -\alpha _1 + h_1 -l_1 + 1 } ; q )_{n } \sum _{\ell =0}^{n } \frac{(q^{ -\lambda  -\alpha _1 + 1 } ;q)_{\ell }}{(q;q )_{n- \ell } (q;q)_{\ell }} q^{-\ell ( 2 n - \ell -1)/2}  (- q^{ \lambda + \alpha _1 - h_1 + l_2 - 1  }t_2 /t_1 ) ^{\ell } , \nonumber \\
& g_4(x) = x^{-\alpha _1 } \sum _{n=0}^{\infty} ( q^{1/2} x^{-1})^n ( q^{\lambda +\alpha _1 } ; q )_n \sum _{k=0}^n  \frac{ (q^{  \lambda +\alpha _1  -h_1 +l_1 } ;q)_{n-k }}{(q;q )_k (q;q)_{n-k}}  (q^{l_1 } t_1)^k (q ^{l_2 } t_2 ) ^{n-k} \nonumber 
\end{align}
are solutions of Eq.~(\ref{C2-Heun}).
\end{prop}
We can show similar properties for the biconfluent equations (Eq.~(\ref{WC-Heun}) and Eq.~(\ref{WC2-Heun})).
\begin{prop}
(i) Assume that function $y(x)$ satisfies Eq.~(\ref{WC-Heun}), which is a variant of the confluent $q$-hypergeometric equation of type $(0,2)$.
Then the function $u(x)= (q^{-h_1 +1/2} t_1^{-1} x;q)_{\infty } (q^{-h_2 +1/2} t_2^{-1} x;q)_{\infty } y(x ) $ is a solution of the equation
\begin{align}
& q^{ 2\lambda +1} g(x/q) + (1-q^{-l_1+1/2}t_1^{-1}x)(1-q^{-l_2+1/2}t_2^{-1} x) g(qx) \label{WC2-Heun01} \\
&-[ q^{ - \tilde{\alpha }_1 -l_1 - l_2 +1 }t_1^{-1} t_2^{-1} x^2 - q^{ \lambda +1/2} ( q^{-l_2} t_2^{-1} + q^{-l_1} t_1^{-1} ) x + q^{ \lambda } (q + 1) ] g(x) =0 \nonumber
\end{align}
where
\begin{align}
& l_1= h_1,\; l_2 =h_2, \; \tilde{\alpha }_1 + \alpha _1 = 1 - 2\lambda .
\label{eq:propCparam2}
\end{align}
Note that Eq.~(\ref{WC2-Heun01}) is a variant of the confluent $q$-hypergeometric equation of type $(2,0)$ (see Eq.~(\ref{WC2-Heun})).
Conversely, if the function $u(x) $ satisfies Eq.~(\ref{WC2-Heun01}), then the function $y(x) = u(x) / ( ( q^{-l_1 +1/2} t_1^{-1} x ; q )_{\infty } (q^{-l_2 +1/2} t_2^{-1} x ; q)_{\infty } ) $ satisfies Eq.~(\ref{WC-Heun}) where the relationship among the parameters is given by Eq.~(\ref{eq:propCparam2}).\\
(ii) The functions
\begin{align}
& g_1 (x) = x^{-\alpha _1 } \sum _{n =0}^{\infty} ( q^{ - \lambda - \alpha _1 + 1/2} x^{-1} )^{n } ( q^{\lambda + \alpha _1 } ; q )_{n } \sum _{\ell =0}^{n }  \frac{q^{- \ell (n -\ell )} ( q^{ h_1 } t_1 )^{n -\ell } (q^{ h_2 } t_2 ) ^{\ell } }{(q;q )_{n -\ell  } (q;q)_{\ell }} , \\
 & g_2 (x) = ((q^{-h_1 +1/2} t_1^{-1} x;q)_{\infty } (q^{-h_2 +1/2} t_2^{-1} x;q)_{\infty } )^{-1} x^{2\lambda + \alpha _1 -1 } \nonumber \\
& \qquad \qquad \cdot  \sum _{n=0}^{\infty} (q^{1/2} x^{-1})^n ( q^{- \lambda - \alpha _1 -1} ; q )_n \sum _{k=0}^n  \frac{ (q^{h_1 } t_1)^k (q ^{h_2 }t_2 ) ^{n-k} }{(q;q )_k (q;q)_{n-k}} \nonumber
\end{align}
are solutions of Eq.~(\ref{WC-Heun}).\\
(iii)
The functions
\begin{align}
& g_3(x) = (q^{-l_1 +1/2} t_1^{-1} x;q)_{\infty } (q^{-l_2 +1/2} t_2^{-1} x;q)_{\infty } x^{2\lambda + \alpha _1 -1 } \\
& \qquad \qquad \cdot \sum _{n =0}^{\infty} ( q^{ \lambda + \alpha _1 - 1/2} x^{-1} )^{n } ( q^{\lambda + \alpha _1 } ; q )_{n } \sum _{\ell =0}^{n }  \frac{q^{- \ell (n -\ell )} ( q^{ l_1 } t_1 )^{n -\ell } (q^{ l_2 } t_2 ) ^{\ell } }{(q;q )_{n -\ell  } (q;q)_{\ell }} , \nonumber \\
 & g_4(x) = x^{-\alpha _1 } \sum _{n=0}^{\infty} (q^{1/2} x^{-1})^n ( q^{\lambda +\alpha _1 } ; q )_n \sum _{k=0}^n  \frac{ (q^{l_1 } t_1)^k (q ^{l_2 }t_2 ) ^{n-k} }{(q;q )_k (q;q)_{n-k}} \nonumber 
\end{align}
are solutions of Eq.~(\ref{WC2-Heun}).
\end{prop}

\section{Limit to the differential equation} \label{sec:limdifftial}

We take the continuum limit $q \to 1 $ from the difference equations given in this paper.
Recall that the variant of the confluent $q$-hypergeometric equation of type $(1,2)$ was given by
\begin{align}
\label{C-Heun02}
&q^{ h_1+h_2-l_1-2\lambda +1/2}t_2(q^{l_1 - 1/2}t_1-x)g(qx)\\
& +(x-q^{h_1 + 1/2} t_1)(x-q^{h_2 + 1/2}t_2)g(x/q) \nonumber \\
&-[q^{\alpha _1}x^2 -q^{h_1+h_2-\lambda +1/2}(q^{-h_2}t_1+q^{-h_1}t_2+q^{-l_1}t_2)x \nonumber \\
& \quad +q^{h_1+h_2-\lambda }(q+1) t_1t_2]g(x)=0. \nonumber 
\end{align}
Set $q=1+ \varepsilon $, $t_2= 1/(T \varepsilon ) $ and consider the limit $\varepsilon \to 0$, which is equivalent to $q\to 1 $.
By using Taylor's expansion
\begin{align}
& g(x/q) =g(x) + (-\varepsilon +\varepsilon ^2 )xg'(x) + \varepsilon ^2 x^2 g''(x) /2 +O( \varepsilon ^3),\\
& g(qx) =g(x) + \varepsilon  x g'(x) + \varepsilon ^2 x^2 g''(x) /2 +O( \varepsilon ^3), \nonumber 
\end{align}
we find the following limit as $\varepsilon  \to 0$:
\begin{align}
& x^2 (x -t_1 ) g''(x) + x \{ T x (x-t_1 ) + ( h_1 -l_1 + 1 )x -2 \lambda ( x - t_1 ) \} g'(x) \label{eq:Cdiffeq} \\
& + \{ \alpha _1 T x^2 +  \lambda ( t_1 T -h_1 + l_1 +\lambda ) x  - \lambda (\lambda +1 ) t_1 \}  g(x)=0 . \nonumber 
\end{align}
This equation has an irregular singularity at $x= \infty $ and regular singularities at $x=0, t_1$, and we can show that the singularity $x=0$ is apparent.
Set $f(x)= x^{ \lambda } g(x)$.
Then we essentially obtain Kummer's confluent hypergeometric equation.
\begin{align}
& (x -t_1 ) f''(x) + \{ (x - t_1 ) T  + h_1 -l_1 + 1 \} f'(x) + (\lambda + \alpha _1 ) T f(x)=0 . \label{eq:Cdiffeqgt} 
\end{align}
By the limit to the differential equation, some solutions of Eq.~(\ref{C-Heun02}) may converge to the solutions of Eq.~(\ref{eq:Cdiffeq}).
For example, the function 
\begin{equation}
x^{\lambda }\  _3\phi_2 (q^{\lambda + \alpha_1},0, x/(q^{l_1-1/2}t_1) ;q^{h_1-l_1+1},q^{h_2-l_1+1}t_2/t_1;q , q ) \end{equation}
in Eq.~(\ref{eq:solssimpconf01}), which is a solution of Eq.~(\ref{C-Heun02}), converges to the function 
\begin{align}
& x^{\lambda } { } _1 F_1 (\lambda +\alpha _1, h_1-l_1+1; T(t_1 -x ) ) \Big( = x^{\lambda } \sum^{\infty}_{n=0} \frac{(\lambda +\alpha_ 1 )_n }{(h_1-l_1+1 )_n n! } T^n (t_1-x) ^n \Big)
\end{align}
for each component of the series as $\varepsilon \to 0$ where $q=1+ \varepsilon $ and $t_2= 1/(T \varepsilon ) $.

We can also obtain Eq.~(\ref{eq:Cdiffeq}) from the variant of the confluent $q$-hypergeometric equation of type $(2,1)$.
Namely we can obtain Eq.~(\ref{eq:Cdiffeq}) from Eq.~(\ref{C2-Heun}) as $\varepsilon \to 0$ by setting $q=1+ \varepsilon $ and $t_2= -1/(T \varepsilon ) $.

We consider the limit from variants of the biconfluent $q$-hypergeometric equation.
Recall that the variant of the confluent $q$-hypergeometric equation of type $(0,2)$ was given by
\begin{align}
\label{WC-Heun02}
& g(qx)+   q^{ 2\lambda +1} ( t_1 ^{-1} q^{-h_1 -1/2} x-1 )(t_2 ^{-1} q^{-h_2 -1/2 } x- 1)g(x/q)\\
&-[t_1 ^{-1} t_2 ^{-1} q^{\alpha _1 + 2\lambda -h_1- h_2 } x^2 - q^{ \lambda +1/2}(q^{-h_2}t_2 ^{-1}+q^{-h_1}t_1 ^{-1})x
+q^{ \lambda } (q + 1) ]g(x)=0 . \nonumber 
\end{align}
Set $q=1+ \varepsilon $, $t_1 ^{-1} = B \varepsilon  ^{1/2} $, $t_2 ^{-1} = -B \varepsilon  ^{1/2} $, and consider the limit $\varepsilon \to 0$.
Then we find the following limit as $\varepsilon  \to 0$:
\begin{align}
& x^2 g''(x) + x ( B ^2 x^2 -2 \lambda ) g'(x)  + \{ \alpha _1 B^2 x^2  + \lambda ( \lambda +1)  \}  g(x)=0 . \label{eq:WCdiffeq}
\end{align}
Set $f(x)= x^{ \lambda } g(x)$.
Then we have
\begin{align}
& f''(x) +  B ^2 x f'(x) + ( \alpha _1 + \lambda )B^2 g(x)=0 ,
\label{eq:WCdiffeqgt}
\end{align}
which is essentially the Hermite-Weber differential equation.
It seems that the solution of Eq.~(\ref{WC-Heun02}) given in Eq.~(\ref{eq:solbiconf00}) tends to the formal series
\begin{align}
& x^{-\alpha _1} \sum _{n=0}^{\infty } \frac{(\lambda +\alpha _1)_{2n}}{n! B^{2n} } x^{-2n} 
\end{align}
for each component as $\varepsilon \to 0$.
This series is a formal solution of Eq.~(\ref{eq:WCdiffeq}).

We can obtain Eq.~(\ref{eq:WCdiffeq}) from the variant of the confluent $q$-hypergeometric equation of type $(2,0)$.
Namely we obtain Eq.~(\ref{eq:WCdiffeq}) from Eq.~(\ref{WC2-Heun}) as $\varepsilon \to 0$ by setting $q=1+ \varepsilon $, $t_1 ^{-1} = B (-\varepsilon  )^{1/2} $ and $t_2 ^{-1} = -B (- \varepsilon  )^{1/2} $.

\section{Concluding remarks} \label{sec:concl}

In this papar, we investigated degenerations of the variant of $q$-hypergeometric equation of degree two and obtained several formal solutions of the variant of singly confluent and biconfluent $q$-hypergeometric equations.
Convergence or divergence of the formal solutions and the resummation should be clarified in a near future.
We propose a problem for obtaining further degeneration of the variant of biconfluent $q$-hypergeometric equation.
We may define a variant of triconfluent $q$-hypergeometric equation by considering the degeneration such that the coefficient of $g(x/q)$ in Eq.~(\ref{WC-Heun}) is a linear polynomial, although we do not know how to obtain the limit to the differential equation (e.g.~Airy's differential equation).

In \cite{Oym}, Ohyama proposed a coalescent diagram of $q$-special functions starting from Heine's $q$-hypergeometric function, which is related with special solutions of $q$-difference Painlev\'e equations.
Although the degenerations of the variant of $q$-hypergeometric equation in this paper are different objects from the ones in Ohyama's paper, it would be interesting to find a unified theory.
For example, the standard singly confluent $q$-hypergeometric equation $(c - a x)u(q x) - (c + q - x)u( x) + qu(x/q) = 0 $ can be essentially obtained from the variant of the confluent $q$-hypergeometric equation of type $(2,1)$ (i.e.~Eq.~(\ref{C2-Heun})) by the limit $t_1 \to 0$.
However it would not be simple for the biconfluent case.
The monograph by Koekoek-Lesky-Swarttouw \cite{KLS} might be useful for further studies.

\section*{Acknowledgements}
The authors are grateful to the referee for valuable comments.
The third author would like to thank Professor Yousuke Ohyama for discussions.
He was supported by JSPS KAKENHI Grant Number JP18K03378.


\begin{thebibliography}{9999}
\bibitem{GR}
G.~Gasper, M.~Rahman. Basic hypergeometric series. Vol. 96. Cambridge university press, 2004.
\bibitem{Hahn}
W.~Hahn, On linear geometric difference equations with accessory parameters, {\it Funkcial. Ekvac.} {\bf 14} (1971), 73--78.
\bibitem{HMST}
N.~Hatano, R.~Matsunawa, T.~Sato, K.~Takemura, Variants of $q$-hypergeometric equation, to appear in {\it Funkcial. Ekvac.}, {\sf arXiv:1910.12560}.
\bibitem{KLS}
R.~Koekoek, P.~A.~Lesky, R.~F.~Swarttouw, Hypergeometric orthogonal polynomials and their $q$-analogues, Springer Monographs in Mathematics. Springer-Verlag, Berlin, 2010.
\bibitem{Oym}
Y.~Ohyama, A unified approach to $q$-special functions of the Laplace type, {\sf arXiv:1103.5232}.
\bibitem{TakR}
K.~Takemura, Degenerations of Ruijsenaars-van Diejen operator and $q$-Painleve equations, {\it J. Integrable Systems} {\bf 2} (2017), xyx008.
\bibitem{TakqH}
K.~Takemura, On $q$-deformations of the Heun equation, {\it SIGMA} {\bf 14} (2018), paper 061.
\end{thebibliography}
\end{document}